\documentclass[12pt,letterpaper,reqno]{amsart}

\usepackage{times}
\usepackage[T1]{fontenc}
\usepackage{mathrsfs}
\usepackage{latexsym}
\usepackage[dvips]{graphics}
\usepackage{epsfig}
\usepackage{amsmath,amsfonts,amsthm,amssymb,amscd, hyperref}
\input amssym.def
\input amssym.tex

\newcommand{\supp}{{\rm supp}}

\newcommand{\hphi}{\widehat{\phi}}  

\newcommand\be{\begin{equation}}
\newcommand\ee{\end{equation}}
\newcommand\bea{\begin{eqnarray}}
\newcommand\eea{\end{eqnarray}}
\newcommand\bi{\begin{itemize}}
\newcommand\ei{\end{itemize}}
\newcommand\ben{\begin{enumerate}}
\newcommand\een{\end{enumerate}}
\newcommand\bc{\begin{center}}
\newcommand\ec{\end{center}}
\newcommand\ba{\begin{array}}
\newcommand\ea{\end{array}}

\newcommand{\R}{\ensuremath{\mathbb{R}}}

\newtheorem{thm}{Theorem}[section]

\newtheorem{lem}[thm]{Lemma}
\newtheorem{prop}[thm]{Proposition}

\newtheorem{defi}[thm]{Definition}

\newtheorem{rek}[thm]{Remark}

\newcommand{\twocase}[5]{#1 \begin{cases} #2 & \text{{\rm #3}}\\ #4
&\text{{\rm #5}} \end{cases}   }

\newcommand{\threecase}[7]{#1 \begin{cases} #2 & \text{{\rm #3}}\\ #4
&\text{{\rm #5}}\\ #6 & \text{{\rm #7}} \end{cases}   }

\numberwithin{equation}{section}

\begin{document}

\title[Low-lying Zeros of Number Field $L$-functions]{Low-lying Zeros of Number Field $L$-functions}

\author{Steven J. Miller}\email{Steven.J.Miller@williams.edu}
\address{Department of Mathematics and Statistics, Williams College, Williamstown, MA 01267}

\author{Ryan Peckner}\email{rpeckner@berkeley.edu}
\address{Department of Mathematics, University of California, Berkeley, CA 94720}


\date{\today}

\thanks{This work was done at the 2009 SMALL Undergraduate Research Project at Williams College, funded by NSF GRANT DMS0850577 and Williams College; it is a pleasure to thank them for their support. We also thank Michael Rosen and the participants of the 2009 YMC at Ohio State for many enlightening conversations. The first named author was also partly supported by NSF grant DMS0600848.}

\begin{abstract} One of the most important statistics in studying the zeros of $L$-functions is the 1-level density, which measures the concentration of zeros near the central point. Fouvry and Iwaniec \cite{FI} proved that the 1-level density for $L$-functions attached to imaginary quadratic fields agrees with results predicted by random matrix theory. In this paper, we show a similar agreement with random matrix theory occurring in more general sequences of number fields. We first show that the main term agrees with random matrix theory, and similar to all other families studied to date, is independent of the arithmetic of the fields. We then derive the first lower order term of the 1-level density, and see the arithmetic enter.
\end{abstract}


\maketitle

\setcounter{equation}{0}


\section{Introduction}

\subsection{Background}

While studying class numbers in the early 1970s, Montgomery made the remarkable observation that the zeros of the Riemann zeta function appear to be correlated in precisely the same way as the eigenvalues of Gaussian random matrices \cite{Mon}. This was based on a chance encounter with Freeman Dyson, who had calculated the eigenvalue pair correlation function for the Gaussian Unitary Ensemble and found it to be
\[
1 - \left(\frac{\sin \pi x}{\pi x}\right)^{2},
\]
exactly the distribution conjectured by Montgomery for the zeros of the zeta function. Extensive numerical computations by Odlyzko \cite{Od1,Od2} support this unexpected correspondence to impressive heights on the critical line.

Attempts to explain this connection rigorously in the number field case have thus far been unsuccessful. However, groundbreaking theoretical work
by Katz and Sarnak has put this goal within reach in the function field setting. They proved that, as one averages over the zeros of suitable families of $L$-functions obtained from geometry, the scaling limit of the spacing measures of the normalized zeros tends to a `universal measure' which is the limit of the spacing measures of the eigenvalues of Gaussian random matrices (see \cite{KaSa1,KaSa2} for details, as well as the survey article \cite{FM} for a description of the development of random matrix theory from nuclear physics to number theory). Moreover, their work predicts that associated to an appropriate family $\mathcal{E}$ of elliptic curves over $\mathbb{Q}$ is a classical compact matrix group $\mathcal{G}(\mathcal{E})$ (which may be viewed as a group of random matrices under normalized Haar measure) in such a way that for any compactly supported even Schwartz function $h$ on $\mathbb{R}$, we have
\begin{equation}\label{eq: density conjecture}
\int_{\mathbb{R}}h(x)W_{\mathcal{G}(\mathcal{E})}(x)dx = \lim_{X \to \infty}\left(1/\sum_{n \leq X}{|\mathcal{E}_{n}|}\right)\sum_{n \leq X, E \in \mathcal{E}_{n}}h\left(\gamma_{E,j}\frac{\log(N_{E})}{2\pi}\right)
\end{equation}
where $N_{E}$ denotes the conductor of the curve $E$,
\[
\mathcal{E}_{n} = \{\mathbb{Q}-\!\text{isogeny classes of} \ E \in \mathcal{E} : N_{E} = n\}
\]
and $1/2 \pm i\gamma_{E,j}$ are the zeros of $L(S, E/\mathbb{Q})$ (normalized to have functional equation $s \to 1-s$). The distribution $W_{\mathcal{G}(\mathcal{E})}$ is canonically associated to the scaling limit of a classical compact group, and gives the density of the normalized spacings between the eigenangles. Katz and Sarnak \cite{KaSa1,KaSa2} showed that for test functions $\phi$ with Fourier transforms supported in $(-1,1)$, the one-level densities of the scaling limits of the classical compact groups are given by\footnote{For the purposes of this paper, the following formulas suffice as we only need to know the one-level densities when $\supp(\hphi) \subset (-1,1)$. See \cite{KaSa1,KaSa2} for determinantal formulas for the $n$-level densities for arbitrary support.}
\begin{eqnarray}\label{eq: eigenvalue distributions}
\int{\phi(x)W_{\rm SO(even)}(x)}dx & = & \widehat{\phi}(0) + \frac{1}{2}\phi(0) \nonumber\\
\int{\phi(x)W_{\rm SO(odd)}(x)}dx & = & \widehat{\phi}(0) + \frac{1}{2}\phi(0) \nonumber\\
\int{\phi(x)W_{\rm O}(x)}dx & = & \widehat{\phi}(0) + \frac{1}{2}\phi(0) \nonumber\\
\int{\phi(x)W_{\rm USp}(x)}dx & = & \widehat{\phi}(0) - \frac{1}{2}\phi(0) \nonumber\\
\int{\phi(x)W_{\rm U}(x)}dx & = & \widehat{\phi}(0).
\end{eqnarray}
The quantity on the right side of $\eqref{eq: density conjecture}$, which due to the normalization by $\displaystyle\frac{\log(N_{E})}{2\pi}$ measures the low-lying zeros of the $L$-functions, is known as the 1-level density for the family. Thus, this conjecture is often referred to as the `density conjecture'.

One expects that an analogue of this conjecture should hold for all suitable families of automorphic $L$-functions, not just those associated to elliptic curves. Indeed, the density conjecture has been verified (up to small support) for a wide variety of families, including all Dirichlet characters, quadratic Dirichlet characters, elliptic curves, weight $k$ level $N$ cuspidal newforms, symmetric powers of GL(2) $L$-functions, and certain families of GL(4) and GL(6) $L$-functions; see \cite{DM1, DM2, HR, HM, ILS, KaSa2, Mil3, OS, RR1, Ro, Rub, Yo2}. We have two goals in this paper. The first is to verify the density conjecture for as large a class of test functions as possible for $L$-functions coming from a patently different situation than that of elliptic curves, namely the $L$-functions of ideal class characters of number fields. As in all other families studied to date, the main term is independent of the arithmetic of the family. Our second goal is to see the effects of the arithmetic in the lower order terms, thereby distinguishing different families.

To make things precise, let $\mathcal{F}$ be a family of number fields, and define for each field $K \in \mathcal{F}$ the $1$-level density
\be
D_{\widehat{\mathcal{CL}(K)}}(\phi)\ =\ \frac{1}{h_{K}}\sum_{\chi \in \widehat{\mathcal{CL}(K)}}\sum_{\substack{\gamma_{\chi} \\ L(1/2 + i\gamma_{\chi},\chi) = 0}}\phi\left(\gamma_{\chi}\frac{\log\Delta_{K}}{2\pi}\right)
\ee
where $\phi$ is an even Schwartz function whose Fourier transform has compact support, $h_K$ is the class number of $K$, $\Delta_{K}$ is the absolute value of its discriminant, and the outer sum runs over the characters of the ideal class group $\mathcal{CL}(K)$ of $K$. Again, due to the rapid decay of $\phi$ and the scaling factor $\frac{\log\Delta_{K}}{2\pi}$, only the low-lying zeros contribute to this sum in the limit as $\Delta_{K} \to \infty$. Since for a given number $X$  there are only finitely many number fields of (absolute value of) discriminant less than $X$, the discriminant must tend to infinity in any infinite family of number fields. Moreover, ordering the family $\mathcal{F}$ according to the increasing parameter $\Delta_K$,  we may consider the limit
\[
D_{\mathcal{F}}(\phi) = \lim_{\Delta_{K} \to \infty}D_{\widehat{\mathcal{CL}(K)}}(\phi),
\]
and this is independent of rearranging fields which have the same value of $\Delta_{K}$. However, there is no good reason to expect this limit to exist if $\mathcal{F}$ is just an arbitrary collection of number fields; thus we reserve the term `family' for a collection $\mathcal{F}$ of number fields whose members have similar arithmetic properties and for which the $1$-level density actually exists. This is by no means an attempt at an actual definition of the term `family', which is an ongoing subject, but it suffices for our purposes, wherein the common arithmetic origin of our fields will be obvious.

Among the wide variety of families for which the density conjecture has been investigated, few have arisen from the number field context. In fact, to our knowledge, the only work to date analyzing the 1-level density for Hecke characters is that of Fouvry-Iwaniec \cite{FI}, who showed that, in the notation above, the $1$-level density $D_{\mathcal{F}}(\phi)$ for $\mathcal{F}$ the family $\mathbb{Q}(-D)$ with $-D$ a fundamental discriminant is given by the symplectic distribution. In addition, recent unpublished work of Andrew Yang \cite{Ya} indicates that the 1-level density for the Dedekind zeta functions of cubic fields is governed by the symplectic distribution. In this paper, we extend the results of \cite{FI} to the family of all CM-fields over a fixed totally real field (see below for definitions). Since infinitely many such families exist, we also derive the first lower order term of the 1-level density (under certain conditions), which allows us to distinguish different families by their arithmetic.

\subsection{1-level density}

In this paper, $K$ will denote a number field of fixed degree $N$ over $\mathbb{Q}$, $h_K$ its class number, $\Delta_K$ the absolute value of its discriminant, $r_1$ and $r_2$ the number of real resp. half the number of complex embeddings\footnote{Thus $r_1 + 2r_2=N$.}, and $R_K$ the regulator.\\ \ \\ \texttt{\emph{Although $K$ will vary, we will generally omit the subscripts from our notation; thus} $h = h_K$, \emph{et cetera.}}\ \\ \ \\ Let $\chi$ be a character of the ideal class group of $K$, and let $\phi$ be an even function in the Schwartz space $\mathcal{S}(\mathbb{R})$ such that the function $\widehat{\phi}$ has compact support; here $\widehat{\phi}$ represents the Fourier Transform\footnote{Note other works may use a different normalization, using $e^{-ixy}$ instead of $e^{-2\pi i xy}$.} \be \widehat{\phi}(y) \ = \ \int_{-\infty}^\infty \phi(x) e^{-2\pi i xy}dx. \ee Assume the generalized Riemann hypothesis, so we may write the zeros of $L(s,\chi)$ as $1/2+i\gamma_\chi$ with $\gamma_\chi\in\R$. Then Weil's explicit formula, as simplified by Poitou, reads \cite{P,BDF,La1}
\begin{eqnarray}\label{eq:expformulapoitou}
\displaystyle\sum_{\gamma_{\chi}}{\phi\left(\gamma_{\chi}\frac{\log\Delta}{2\pi}\right)} &= & \frac{1}{\log\Delta}\left[4\delta_{\chi}\int_{0}^{\infty}{\widehat{\phi}\left(\frac{x}{\log\Delta}\right)\text{cosh}(x/2)dx}\right. \nonumber\\  && + \widehat{\phi}(0)(\log\Delta - N\gamma_{{\rm EM}} -N\log8\pi - \displaystyle\frac{r_{1}\pi}{2}) \nonumber\\
&& - \displaystyle\sum_{\mathfrak{p}}\log N\mathfrak{p}\sum_{m = 1}^{\infty}\frac{\widehat{\phi}\left(m\displaystyle\frac{\log N\mathfrak{p}}{\log\Delta}\right)}{N\mathfrak{p}^{m/2}}(\chi(\mathfrak{p})^{m} + \chi(\mathfrak{p})^{-m}) \nonumber\\
&& + r_{1}\int_{0}^{\infty}{\frac{\widehat{\phi}(0) - \widehat{\phi}(x)}{2\text{cosh}(x/2)}dx} + \left.N\int_{0}^{\infty}{\frac{\widehat{\phi}(0) - \widehat{\phi}(x)}{2\text{sinh}(x/2)}dx} \right],\ \ \
\label{eq:expform2}
\end{eqnarray}
where the sum on the left is over the imaginary parts $\gamma_{\chi}$ of the zeros of $L(s,\chi)$, the sum on the right is over the prime ideals of the ring of integers of $K$, $\gamma_{{\rm EM}}$ is the Euler-Mascheroni constant and $\delta_{\chi}$ is the indicator of the trivial character (i.e., it is 1 if $\chi$ is the trivial character and 0 otherwise). As is standard, we rescaled the zeros by $\log\Delta$ to facilitate applications to studying the zeros near the central point.

We now wish to average this formula over all characters $\chi$ of the ideal class group $\mathcal{CL}(K)$ of $K$. We denote its dual by $\widehat{\mathcal{CL}(K)}$, and note that its cardinality is the class number $h$. By $\chi(\mathfrak{p})$ we of course mean the value of $\chi$ on the ideal class of $\mathfrak{p}$.  For any non-zero integer $m$ and any prime $\mathfrak{p}$ of $K$ we have
\be \displaystyle\sum_{\chi \in \widehat{\mathcal{CL}(K)}}\chi(\mathfrak{p})^m\ =\ \Bigg\{ \begin{array}{rl}
  h &\text{if} \ \mathfrak{p} \ \text{is principal} \\
  h &\text{if} \ \mathfrak{p} \ \text{is not principal and} \ m\,|\,\text{ord}_{\mathcal{CL}(K)}(\mathfrak{p}) \\
  0 &\text{otherwise.}
       \end{array} \ee
Averaging the explicit formula over the family yields the one-level density
\begin{eqnarray}\label{eq:average}
& & D_{\widehat{\mathcal{CL}(K)}}(\phi) \ := \ \frac{1}{h}
\sum_{\chi \in \widehat{\mathcal{CL}(K)}}
\sum_{\substack{\gamma_{\chi} \\ L(1/2 + i\gamma_{\chi},\, \chi) = 0}}
\phi\left(\gamma_{\chi}\frac{\log\Delta}{2\pi}\right)\nonumber\\
&= & \frac{1}{\log\Delta}
\Bigg[ \frac{4}{h}\int_{0}^{\infty}{\widehat{\phi}\left(\frac{x}{\log\Delta}\right)\cosh\left(\frac{x}{2}\right)dx}
+  \widehat{\phi}(0)\cdot\left(\log \Delta - N\gamma_{{\rm EM}} -N\log 8\pi - \displaystyle\frac{r_{1}\pi}{2}\right)  \nonumber\\
& & -  2 \left( \displaystyle\sum_{\mathfrak{p} \ {\rm non-principal}}\log N\mathfrak{p}\sum_{\substack{m \geq 1\\ \mathfrak{p}^m \ \text{principal}}}  {\frac{\widehat{\phi}\left( m \frac{\log N\mathfrak{p}}{\log \Delta}\right)}{N\mathfrak{p}^{m/2}}}
 +  \sum_{\substack{\mathfrak{p} \ \text{principal}}}{\log N\mathfrak{p}}\sum_{m=1}^{\infty}{\frac{\widehat{\phi}\left(m \frac{\log N\mathfrak{p}}{\log \Delta}\right)}{N\mathfrak{p}^{m/2}}}  \right) \nonumber\\
&  & +  r_{1}\int_{0}^{\infty}{\frac{\widehat{\phi}(0) - \widehat{\phi}(x)}{2\cosh(x/2)}dx} + N\int_{0}^{\infty}{\frac{\widehat{\phi}(0) - \widehat{\phi}(x)}{2\sinh(x/2)}\ dx} \Bigg].
\end{eqnarray}
We wish to ascertain the behavior of this average as $\Delta \rightarrow \infty$.

We recall some relevant facts from algebraic number theory (see Chapter 4, Part 1 of \cite{La1} or \cite{Wa} for more details). A number field $K_0$ is called totally real if every embedding of $K_0$ into $\mathbb{C}$ has image contained in $\mathbb{R}$, i.e. $K_0$ is generated over $\mathbb{Q}$ by an algebraic number all of whose conjugates are real. On the other hand, a number field $K$ is called totally imaginary if \textit{no} embedding of $K$ into $\mathbb{C}$ has image contained in $\mathbb{R}$. A \textit{CM-field} is a totally imaginary number field which forms a quadratic extension of a totally real number field. This totally real field is unique and is denoted $K^+$. $K$ then takes the form $K = K^{+}(\sqrt{\beta})$, where $\beta$ is a square-free element of $\mathcal{O}_{K^+}$ which is totally negative, e.g. $\sigma(\beta) < 0$ for every embedding $\sigma: K^+ \hookrightarrow \mathbb{R}$. Any totally real field obviously has infinitely many CM-fields over it, and CM-fields form a rich and abundant class of number fields. Indeed, any finite abelian extension of $\mathbb{Q}$ is either totally real or is a CM-field (by the Kronecker-Weber theorem), and the abbreviation CM reflects the strong connection between CM-fields and the theory of abelian varieties with complex multiplication (see IV.18 of \cite{S} for details).

We now describe our family of number fields. Fix a totally real number field $K_{0}/\mathbb{Q}$ of class number one and degree $N$ over $\mathbb{Q}$, and let $\{K_{\Delta}\}$ be the family of all CM-fields for which $K_{\Delta}^+ = K_{0}$, ordered by (absolute value of) discriminant $\Delta$. Although it may be the case that several $K$ share the same value of $\Delta$, there are by standard results only finitely many which do (\cite{La1}, pg. 121), so their ordering is irrelevant. Each of these fields has degree $2N$ over $\mathbb{Q}$. We denote the class number of $K_\Delta$ by $h_\Delta$.

Define distributions $S_{1}(\Delta, \cdot), S_{2}(\Delta, \cdot)$ by
\begin{eqnarray}
S_{1}(\Delta, \phi) & \ := \ & -2\displaystyle\sum_{\mathfrak{p} \ {\rm non-principal}}\log N\mathfrak{p}\sum_{\substack{m \geq 2\\ \mathfrak{p}^m \ \text{principal}}}  {\frac{\widehat{\phi}\left(m \frac{\log N\mathfrak{p}}{\log \Delta}\right)}{N\mathfrak{p}^{m/2}}} \nonumber\\
S_{2}(\Delta, \phi) & := & -2\displaystyle\sum_{\mathfrak{p} \ {\rm principal}}\log N\mathfrak{p}\sum_{m = 1}^{\infty} {\frac{\widehat{\phi}\left(m \frac{\log N\mathfrak{p}}{\log \Delta}\right)}{N\mathfrak{p}^{m/2}}};
\end{eqnarray} note the $m$-sum for $S_1(\Delta,\phi)$ starts at $2$ and not $1$ because $\mathfrak{p}$ is not principal but $\mathfrak{p}^m$ is. In terms of this notation, \eqref{eq:average} yields

\begin{thm}[Expansion for the 1-level density]\label{thm:oneleveldensity} Notation as above, if $\phi$ is an even Schwartz function with $\supp(\widehat{\phi}) \subset (-\sigma, \sigma)$, then
\begin{eqnarray}\label{eq:density}
D_{\widehat{\mathcal{CL}(K)}}(\phi)  &\ :=\ & \displaystyle\frac{1}{h_{\Delta}}
\sum_{\chi \in \widehat{\mathcal{CL}(K_{\Delta})}}
\sum_{\substack{\gamma_{\chi} \\ L(1/2 + i\gamma_{\chi},\, \chi) = 0}}
\phi\left(\gamma_{\chi}\frac{\log\Delta}{2\pi}\right) \nonumber\\
&=& \frac{1}{\log\Delta}
\Bigg[ \frac{4}{h_{\Delta}}\int_{0}^{\infty}{\widehat{\phi}\left(\frac{x}{\log\Delta}\right)\cosh\left(\frac{x}{2}\right)dx} \nonumber\\
 & & \ + \ \widehat{\phi}(0)\cdot\left(\log \Delta - 2N\gamma_{{\rm EM}} -2N\log 8\pi\right) \nonumber\\
 & & \ +\ S_{1}(\Delta, \phi) + S_{2}(\Delta, \phi) + 2N\int_{0}^{\infty}{\frac{\widehat{\phi}(0) - \widehat{\phi}(x)}{2\sinh(x/2)}\ dx} \Bigg]. 
\end{eqnarray}
\end{thm}

Note that we've used $r_1 = 0$, since $K$ is totally imaginary.

\subsection{Main results}

Our first result is the following.

\begin{thm}\label{thm:main} Assume the Generalized Riemann Hypothesis for all Hecke $L$-functions.
Let $\phi$ be an even Schwartz function whose Fourier transform is supported in $(-1,1)$. Fix a normal, totally real number field $K_{0}/\mathbb{Q}$ of class number one and degree $N$ over $\mathbb{Q}$, and let $\{K_{\Delta}\}$ be the family of all CM-fields for which $K_{\Delta}^+ = K_{0}$, ordered by the absolute value of the discriminant $\Delta$. Then
\begin{equation}
D_{\widehat{\mathcal{CL}(K)}}(\phi)\ =\ \widehat{\phi}(0) - \frac{1}{2}\phi(0) + O\left(\frac{\log\log\Delta}{\log\Delta}\right),
\end{equation}
which implies that the one-level density agrees with the scaling limit of symplectic but not unitary or orthogonal matrices (see \eqref{eq: eigenvalue distributions}).
\end{thm}

Frequently in computing 1-level densities of families, we are able to improve our support or isolate lower order terms if we restrict to a sub-family of the original family which is more amenable to averaging. See for instance the results of Gao \cite{Gao} and Miller \cite{Mil4} for sub-families of the family of quadratic Dirichlet characters with even fundamental discriminants at most $X$,\footnote{The sub-family studied is $\{8d : 0 < d \le X$; $d$ an odd, positive square-free
fundamental discriminant$\}$; this extra restriction facilitates the application of Poisson summation.} or \cite{Mil3} for families of elliptic curves. The situation is similar here; to derive the lower order terms of the 1-level density, we make the additional assumption that the class number of $K_0$ in the narrow sense is 1. Recall that the {\em narrow class group} of $K_0$ is defined similarly to the ordinary ideal class group, except that ideals are considered equivalent if and only if they differ by a totally positive element of $K_0$ rather than an arbitrary one. 

By restricting the family of number fields we study a little bit, we are able to isolate the first lower order term, which depends on the arithmetic of the field.


\begin{thm}[First Lower Order Term]\label{thm:lowterm} Assume the Generalized Riemann Hypothesis for all Hecke $L$-functions. Let $\phi$ be an even Schwartz function whose Fourier transform is supported in $(-1,1)$. Fix a normal, totally real number field $K_{0}/\mathbb{Q}$ whose class number in the narrow sense is 1, and let $\{K_{\Delta}\}$ be the family of all CM-fields of odd class number (in the usual sense) for which $K_{\Delta}^+ = K_{0}$, ordered by the absolute value of the discriminant $\Delta$. For a number field $E/\mathbb{Q}$, let $\rho_E$ be the residue of its Dedekind zeta function at the simple pole $s=1$
\be
\rho_{E}\ =\ {\rm res}_{s=1}\zeta_{E}(s)\ =\ \frac{2^{r_1}(2\pi)^{r_2}h_{E}R_{E}}{w_{E}\sqrt{|D_{E/\mathbb{Q}}|}},
\ee
and let $\gamma_E$ denote its Euler constant
\be
\gamma_E\ =\ \frac{d}{ds}[(s-1)\zeta_{E}(s)]_{s=1}\ =\ \lim_{s\to1}\left(\zeta_{E}(s) - \frac{\rho_E}{s-1}\right).
\ee Let $\gamma_{{\rm EM}}$ be the Euler-Mascheroni constant. Then the 1-level density is given by
\be
D_{\widehat{\mathcal{CL}(K)}}(\phi)\ =\ \widehat{\phi}(0)  - \frac{1}{2}\phi(0) + \frac{1}{\log\Delta}\left(\widehat{\phi}(0)\tau(\Delta) + \mathcal{L}_1(\Delta)\right) + O\left(\frac{1}{\log^{2}\Delta}\right)
\ee
where
\begin{eqnarray}
\mathcal{L}_1(\Delta) & = &
\frac{4}{h_{\Delta}}\int_{0}^{\infty}{\widehat{\phi}\left(\frac{x}{\log\Delta}\right)\cosh\left(\frac{x}{2}\right)dx}
+ \widehat{\phi}(0)\cdot\left(- 2N\gamma_{{\rm EM}} -2N\log 8\pi\right) \nonumber\\
& &\ + 2N\int_{0}^{\infty}{\frac{\widehat{\phi}(0) - \widehat{\phi}(x)}{2\sinh(x/2)}\ dx}
\end{eqnarray}
and
\be
\tau(\Delta)\ =\ 4\frac{\gamma_{K_0}}{\rho_{K_0}} - 2\frac{\gamma_K}{\rho_K} -4\sum_{\substack{\mathfrak{q}\subset\mathcal{O}_{K_0} \\ {\rm inert \ in} \ K}}\frac{\log N\mathfrak{q}}{N\mathfrak{q}^{2}-1}.
\ee
Moreover, $\tau(\Delta) = O(1)$, with the implied constant depending on $K_0$.

\end{thm}

\begin{rek} As is common in many families of $L$-functions (see for example \cite{FI,HKS,Mil2,Mil3,Mil4,Mil5,MilMo,Ya,Yo1}), the main term in the 1-level density is independent of the arithmetic of the family, which only surfaces in the lower order terms.


\end{rek}

\indent This paper is organized as follows. After analyzing part of the first lower order term, we prove a lemma on CM-fields that allows us to bound sums over principal primes of degree 1. We proceed to reduce sums over $K$ to sums over $K_0$, which are then handled using standard algebraic number theory. To deal with sums over degree 2 primes, we introduce a variant of the Dedekind zeta function of $K_0$ and show that integration against its logarithmic derivative yields the desired quantities (up to reasonably small error), from which we obtain the result. In Section \ref{sec:3}, we restrict our class of number fields in order to obtain complete control of the ramification behavior, which allows us to reduce the error terms significantly. We then extract the first lower order term by closely studying the arithmetic of the families in question, in the process proving a discriminant-independent bound on number field Euler constants that we haven't seen elsewhere in the literature (see Proposition \ref{prop:eulerbound} and Appendix \ref{sec:appendix}).

\section{Proof of Theorem \ref{thm:main}}

The proof of Theorem \ref{thm:main} follows from an analysis of the terms in Theorem \ref{thm:oneleveldensity} (the one-level density expansion from averaging the explicit formula over the family). The terms other than $S_i(\Delta,\phi)$ are readily analyzed. To see this, we first need a lemma relating the size of $h_\Delta$ to $\Delta$.

\begin{lem}\label{lem:class number}
We have $\log h_{\Delta} \sim \frac{1}{2}\log\Delta$ as $\Delta \rightarrow \infty$.
\end{lem}
\begin{proof}
Since the fields $K_\Delta$ all have the same degree over $\mathbb{Q}$, we have by the Brauer-Siegel Theorem (\cite{La1}, Chapter XVI) that
\be
\log(h_{\Delta}R_{\Delta})\ \sim\ \frac{1}{2}\log\Delta \ \ \text{as}\ \Delta \rightarrow \infty.
\ee
The regulator $R_{\Delta}$ satisfies (\cite{Wa}, pg. 41)
\be
\frac{R_{\Delta}}{R_{K^+}}\ =\ \frac{1}{Q}2^{N-1}
\ee
where $Q = 1 \ \text{or} \ 2$, and therefore $R_{\Delta}$ is bounded by a constant independent of $\Delta$. This proves the claim.
\end{proof}

\begin{lem}\label{lem:htrigfns} Assume $\supp(\widehat{\phi}) \subset (-\sigma, \sigma)$ with $\sigma < 1$. Then the terms involving $\cosh$ and $\sinh$ in Theorem \ref{thm:oneleveldensity} are $O(1/\log \Delta)$. \end{lem}

\begin{proof} The last two terms, where the hyperbolic trig functions are in the denominator, are readily analyzed. As $\cosh(x/2) \gg 1$ and decays exponentially, the integrand with $\cosh$ in the denominator is $O(1)$. The $\sinh$ integral is handled similarly (note everything is well-behaved near $x=0$ because $\phi$ is differentiable, and by L'Hopital's rule the quotient is bounded near $x=0$).

We are left with handling the integral of $\hphi$ against $\cosh$. Changing variables ($u = x/\log \Delta$) gives \be  \frac{4}{h_{\Delta}\log \Delta}\int_{0}^{\infty}\widehat{\phi}\left(\frac{x}{\log\Delta}\right)\cosh\left(\frac{x}{2}\right)dx \ = \ \frac{4}{h_\Delta} \int_0^\infty \hphi(u) \cosh\left(\frac{u \log \Delta}{2}\right) du. \ee Using $2\cosh(t) = e^t + e^{-t}$, we see this integral is dominated by \be \frac1{h_\Delta} \int_0^\infty \left|\hphi(u)\right| \Delta^{u/2} du \ \ll \ \frac{\sigma \Delta^{\sigma/2}}{h_\Delta}, \ee which tends to zero by Lemma \ref{lem:class number} as $\sigma < 1$. \end{proof}

Thus, by the above lemma, the asymptotic behavior of $\mathcal{F}(\Delta, \phi)$ for fixed $\phi$ is determined by that of $S_{1}$ and $S_{2}$. While the hyperbolic integrals will contribute lower order terms of size $1/\log \Delta$, the values of these integrals are independent of the family.

\ \\
\texttt{\emph{In what follows, we drop $\Delta$ from our number field notation; thus $K= K_{\Delta}$, $h=h_{\Delta}$, et cetera.}}\ \\

Before analyzing $S_1$ and $S_2$, we first prove some lemmas on CM-fields which will be essential in our investigations.

\subsection{Lemmas on CM-fields}\label{sec: CM-fields}

Just as in the case of quadratic fields over $\mathbb{Q}$, one easily proves the following.

\begin{lem}
Let $K$ be a CM-field and $\beta \in \mathcal{O}_{K^+}$ a totally negative, square-free element such that $K = K^{+}(\sqrt{\beta})$. Then either
\begin{eqnarray*}
\mathcal{O}_{K}\ =\ \mathcal{O}_{K^+}[\sqrt{\beta}] & \ \ \text{or} \ \ & \mathcal{O}_{K}\ =\ \mathcal{O}_{K^+}\left[\frac{1+\sqrt{\beta}}{2}\right].
\end{eqnarray*}
\end{lem}
Indeed, the minimal polynomial of an element $\alpha = x + y\sqrt{\beta} \in K, x,y \in K^+$ over $K^+$ is
\[
t^2 - 2xt + x^{2} - \beta y^2
\]
so by transitivity of integral closure, $\alpha \in \mathcal{O}_{K}$ if and only if $2x, x^{2} - \beta y^{2} \in \mathcal{O}_{K^+}$. The two possibilities of the lemma then correspond to whether $x \in \mathcal{O}_{K^+}$ or $x \in \displaystyle\frac{1}{2}\mathcal{O}_{K^+}$.

The following lemma is crucial, as it allows us to bound sums over degree 1 principal primes (by showing the sums are vacuous if the support is restricted as in Theorem \ref{thm:main}).

\begin{lem}\label{lem: key}
Let $K$ be a CM-field with maximal real subfield $K^+$. Choose $\beta \in K^+$ which is totally negative and such that $K = K^+(\sqrt{\beta})$. Let $\mathfrak{p} \subset \mathcal{O}_{K}$ be a principal prime ideal of degree 1 with norm $N\mathfrak{p} = p$. Then $p\geq C\Delta$, where $C$ is a constant depending only on $K^{+}$.
\end{lem}

\begin{proof}
We assume that $\mathcal{O}_{K} = \mathcal{O}_{K^+}[\sqrt{\beta}]$; the other case is similar. We first claim that $p \geq |N^{K^+}_{\mathbb{Q}}(\beta)|$. Since $\mathfrak{p}$ is principal, there exist $x,y \in \mathcal{O}_{K^+}$ such that $\mathfrak{p} = (x + y\sqrt{\beta})$. Suppose $y = 0$; then
\begin{eqnarray*}
N\mathfrak{p} \ :=\ N^{K}_{\mathbb{Q}}(\mathfrak{p}) &\ =\ & N^{K^+}_{\mathbb{Q}}(N^{K}_{K^+}(\mathfrak{p})) \\
& = & N^{K^+}_{\mathbb{Q}}(x^2) \\
& = & N^{K^+}_{\mathbb{Q}}(x)^2
\end{eqnarray*}
which is a contradiction since $p = N\mathfrak{p}$ is a prime number ($|N^{K^+}_{\mathbb{Q}}(x)| > 1$ because $x$ can't be a unit). Thus $y \neq 0$.

Assume now $y \neq 0$. Recall the minimal polynomial of $x+y\sqrt{\beta}$ over $K^+$ is
\be
t^2 - 2xt + x^{2} - \beta y^2,
\ee
so $N^{K}_{K^+}(\mathfrak{p}) = N^{K}_{K^+}(x + y\sqrt{\beta}) = x^{2} - \beta y^2$. Hence, since the degree is multiplicative over towers,
\begin{eqnarray}
p &\ = \ & |N^{K^+}_{\mathbb{Q}}(x^2 - \beta y^2)| \nonumber\\
&=& \left|\prod_{\sigma:K^+ \to \mathbb{C}}{\sigma(x^2 - \beta y^2)}\right|\nonumber\\
& = &\left|\prod_{\sigma:K^+ \to \mathbb{C}}{(\sigma(x)^2 - \sigma(\beta)\sigma(y)^2)}\right|.
\end{eqnarray}
We now use our assumption that $\beta$ is totally negative, which implies that $\sigma(\beta) < 0$ for each $\sigma$. We have $-\sigma(\beta) = |\sigma(\beta)|$ and so
\begin{equation}
\left|\prod_{\sigma}{(\sigma(x)^2 - \sigma(\beta)\sigma(y)^2)}\right|\ =\ \left|\prod_{\sigma}{(\sigma(x)^2 + |\sigma(\beta)|\sigma(y)^2)}\right|.
\end{equation}
Since $x,y \in K^+$ and $K^+$ is totally real, we have $\sigma(x),\sigma(y) \in \mathbb{R}$ for each $\sigma$. Therefore $\sigma(x)^2 \geq 0,\sigma(y)^2 > 0$ and so
\begin{eqnarray}
\left|\prod_{\sigma}{(\sigma(x)^2 + |\sigma(\beta)|\sigma(y)^2)}\right| &\ = \ & \prod_{\sigma}{(\sigma(x)^2 + |\sigma(\beta)|\sigma(y)^2)} \nonumber\\
&\geq & \prod_{\sigma}{\sigma(x)^2} + \prod_{\sigma}{|\sigma(\beta)|\sigma(y)^2}\nonumber\\
&=& N^{K^+}_{\mathbb{Q}}(x)^2 + |N^{K^+}_{\mathbb{Q}}(\beta)| \cdot N^{K^+}_{\mathbb{Q}}(y)^2.
\end{eqnarray}
Since $y \neq 0$ and $y \in \mathcal{O}_{K^+}$, $N^{K^+}_{\mathbb{Q}}(y)^2$ is a positive integer. Thus the last expression is at least $|N^{K^+}_{\mathbb{Q}}(\beta)|$, which proves the claim.

By the relative discriminant formula, and since $[K : K^+] =2$, we have
\begin{equation}
D_{K/\mathbb{Q}}\ =\ N^{K^+}_{\mathbb{Q}}(D_{K/K^{+}}) \cdot D_{K^{+}/\mathbb{Q}}^2
\end{equation}
where for an extension of number fields $K/E$, $D_{K/E}$ denotes the relative discriminant (which we take to be an integer if $E = \mathbb{Q}$, although it is an ideal of $\mathcal{O}_{E}$ in general). Since $D_{K/K^{+}} = (4\beta)$, we have $N^{K^+}_{\mathbb{Q}}(D_{K/K^{+}}) = 4^{N}|N^{K^+}_{\mathbb{Q}}(\beta)|.$ Therefore, by the above claim, we have
\begin{equation}
p\ \geq\ |N^{K^+}_{\mathbb{Q}}(\beta)|\ =\ \frac{|D_{K/\mathbb{Q}}|}{4^{N}D_{K^{+}/\mathbb{Q}}^2}\ =\ \frac{\Delta}{4^{N}D_{K^{+}/\mathbb{Q}}^2}
\end{equation}
Finally, note that $1 / (4^{N}D_{K^{+}/\mathbb{Q}}^2)$ depends only on $K^{+}$.
\end{proof}
In particular, since in our setting $K^{+} = K_{0}$ is fixed, we see that $C$ is independent of $\Delta$. This observation will be crucial in what follows, in that it allows us to assert the vacuity of certain sums since they only involve primes whose norms lie outside the support of $\widehat{\phi}$.

\begin{rek}
The CM structure is crucial to obtain such a strong lower bound on the norm of degree 1 principal primes. In general, the results of Lagarias, Montgomery and Odlyzko \cite{LMO} and Oesterl\'{e} \cite{Oe} guarantee that for $L/K$ a Galois extension of number fields, there exists a prime $\mathfrak{p}$ of $K$ of norm at most $70(\log |D_{L/\mathbb{Q}}|)^2$. One must therefore avoid number fields with extensions of small discriminant in order to obtain such a bound.
\end{rek}

\subsection{Evaluation of $S_{1}$}

\begin{lem}\label{lem: S_1} Assume  $\supp(\hphi) \subset (-\sigma, \sigma)$.
If $\sigma < 1$, we have
\be
S_{1}(\Delta, \phi)\ =\ O(\log\log\Delta) \ \text{as} \ \Delta\rightarrow\infty.
\ee
\end{lem}

\begin{proof} First, we claim that
\be\label{eq:s1deltaphiprincnonprinc}
S_{1}(\Delta, \phi)\ =\ -2\sum_{\substack{\mathfrak{p}\ {\rm non-principal} \\ \mathfrak{p}^2\ {\rm principal}}}\frac{\log N\mathfrak{p}}{N\mathfrak{p}}\widehat{\phi}\left(2\frac{\log N\mathfrak{p}}{\log \Delta}\right) + O(1).
\ee
Indeed, since $\hat{\phi}$ is bounded, and since each rational prime $p$ has at most $2N$ prime ideals lying over it in $K$, the sum
\be
\sum_{\mathfrak{p}\ {\rm non-principal}}\log N\mathfrak{p}\sum_{m=3}^{\infty}\frac{\widehat{\phi}\left( m \frac{\log N\mathfrak{p}}{\log \Delta}\right)}{N\mathfrak{p}^{m/2}}
\ee
is bounded by a constant times a convergent series, namely \be \sum_p \sum_m \frac{\log p}{p^m} \ \ll \ \sum_p \frac{\log p}{p^3} \ \ll \ 1. \ee This proves \eqref{eq:s1deltaphiprincnonprinc}.
\ \\

For $K/E$ an extension of number fields and $\mathfrak{p}$ a prime ideal of $\mathcal{O}_{K}$, we denote by $f_{K/E}(\mathfrak{p})$ the residue degree of $\mathfrak{p}$ over $E$, so that $N^{K}_{E}(\mathfrak{p}) = \mathfrak{q}^{f_{K/E}(\mathfrak{p})}$, where $\mathfrak{q} = \mathfrak{p} \cap \mathcal{O}_{E}$. Notice that
\begin{equation}
\sum_{\substack{\mathfrak{p}\ {\rm non-principal} \\ \mathfrak{p}^2\ {\rm principal}}}\frac{\log N\mathfrak{p}}{N\mathfrak{p}}\widehat{\phi}\left(2  \frac{\log N\mathfrak{p}}{\log \Delta}\right)\ =\ \nonumber\\
\sum_{\substack{\mathfrak{p}\ {\rm non-principal} \\ \mathfrak{p}^2\ {\rm principal} \\ f_{K/\mathbb{Q}}(\mathfrak{p})=1}}\frac{\log N\mathfrak{p}}{N\mathfrak{p}}\widehat{\phi}\left(2 \frac{\log N\mathfrak{p}}{\log \Delta}\right) + O(1)
\end{equation}
since the complementary sum is again bounded up to a constant by the convergent series $\sum_p \frac{\log p}{p^2}$. By the compact support of $\hphi$, we have
\be
\sum_{\substack{\mathfrak{p}\ {\rm non-principal} \\ \mathfrak{p}^2\ {\rm principal} \\ f_{K/\mathbb{Q}}(\mathfrak{p})=1}}\frac{\log N\mathfrak{p}}{N\mathfrak{p}}\widehat{\phi}\left(2  \frac{\log N\mathfrak{p}}{\log \Delta}\right)\ =\ \sum_{\substack{\mathfrak{p}\ {\rm non-principal} \\ \mathfrak{p}^2\ {\rm principal} \\ f_{K/\mathbb{Q}}(\mathfrak{p})=1 \\ \log N\mathfrak{p} < \frac{\sigma\log\Delta}{2}}}\frac{\log N\mathfrak{p}}{N\mathfrak{p}}\widehat{\phi}\left(2 \frac{\log N\mathfrak{p}}{\log \Delta}\right).
\ee
Let $\mathfrak{p}$ be a prime of degree 1 over $\mathbb{Q}$ such that $\mathfrak{p}^2$ is principal, say $\mathfrak{p}^2 = (\alpha)$. Either $\alpha \in \mathcal{O}_{K^+}$ or $\alpha \in \mathcal{O}_{K} \setminus \mathcal{O}_{K^+}$. Denote these contributions by $S_{1,1}(\Delta,\phi)$ and $S_{1,2}(\Delta,\phi)$.

Suppose first that $\alpha \in \mathcal{O}_{K^+}$. Then $\alpha\mathcal{O}_{K^+}$ is a prime ideal of $\mathcal{O}_{K^+}$ since $N_{K/\mathbb{Q}}(\mathfrak{p})^2 = N_{K^{+}/\mathbb{Q}}(\alpha)^2$, and it ramifies in $K$. Therefore, since $f_{K/\mathbb{Q}}(\mathfrak{p}) = 1$ implies that $p =N\mathfrak{p}$ is a rational prime, $p$ ramifies in $K$. As the ramified rational primes in $K$ are precisely those dividing $\Delta$, we find
\begin{eqnarray}
S_{1,1}(\Delta,\phi) &\ :=\ & \sum_{\substack{\mathfrak{p}\ {\rm non-principal} \\ \mathfrak{p}^2 = (\alpha), \alpha\in\mathcal{O}_{K^+} \\ f_{K/\mathbb{Q}}(\mathfrak{p})=1 \\ \log N\mathfrak{p} < \frac{\sigma\log\Delta}{2} \\}}\frac{\log N\mathfrak{p}}{N\mathfrak{p}}\widehat{\phi}\left(2 \frac{\log N\mathfrak{p}}{\log \Delta}\right)\\
& \ \ll \  & \displaystyle\sum_{\substack{p \\ p\,|\,\Delta}}\frac{\log p}{p}\ =\ O(\log\log\Delta),
\end{eqnarray}
where we used the standard fact\footnote{Note $\frac{\log u}{u}$ is decreasing for $u \ge 3$, so the sum is maximized when $\Delta$ is a primorial. If $2\cdot 3 \cdots p_r = \Delta$ then $p_r \sim \log \Delta$, and the claim follows from partial summation.} that $\sum_{p|\Delta}\frac{\log p}{p} \ll \log\log \Delta$.

Now consider the case when $\alpha \in \mathcal{O}_{K}\backslash\mathcal{O}_{K^+}$. Let
\be
S_{1,2}(\Delta,\phi)\ :=\ \sum_{\substack{\mathfrak{p}\ {\rm non-principal} \\ \mathfrak{p}^2 = (\alpha), \alpha\in\mathcal{O}_{K}\backslash\mathcal{O}_{K^+} \\ f_{K/\mathbb{Q}}(\mathfrak{p})=1 \\ \log N\mathfrak{p} < \frac{\sigma\log\Delta}{2} \\}}\frac{\log N\mathfrak{p}}{N\mathfrak{p}}\widehat{\phi}\left(2 \frac{\log N\mathfrak{p}}{\log \Delta}\right).
\ee
In this situation, we have $N_{K/\mathbb{Q}}(\mathfrak{p})^2 = N_{K/\mathbb{Q}}(\alpha)$, so the proof of Lemma \ref{lem: key} shows that $N\mathfrak{p} \geq C\sqrt{\Delta}$, where $C$ is a positive constant independent of $\Delta$. Hence, since $\sigma < 1$, the condition $\log N\mathfrak{p} < \frac{\sigma\log\Delta}{2}$ on the sum implies that $S_{1,2}(\Delta,\phi)$ is zero for sufficiently large $\Delta$. Putting things together, we have for $\sigma < 1$ that
\begin{eqnarray}
S_{1}(\Delta, \phi)  \ = \ S_{1,1}(\Delta,\phi) + S_{1,2}(\Delta,\phi) + O(1) \ = \  O(\log\log\Delta),
\end{eqnarray}
which proves the claim.
\end{proof}

\subsection{Reduction of $S_{2}$}\label{sec: S_2}

In this subsection we replace $S_2$ with sums which are easier to evaluate. We determine those sums in the next subsection, which will complete the analysis of $S_2$.

We write $S_{2}$ as a sum
\begin{equation}\label{eq:diff}
S_{2}(\Delta, \phi)\ =\ S_{2,1}(\Delta, \phi) + S_{2,2}(\Delta, \phi)
\end{equation}
where
\begin{eqnarray}\label{eq:terms}
S_{2,1}(\Delta, \phi) & \ := \ & -2\displaystyle\sum_{\substack{\mathfrak{p} \ \text{principal}}}\log N\mathfrak{p}\sum_{\substack{m \geq 1 \\ (m,h_{\Delta}) = 1}}{\frac{\widehat{\phi}\left( m \frac{\log N\mathfrak{p}}{\log \Delta}\right)}{N\mathfrak{p}^{m/2}}} \nonumber\\
S_{2,2}(\Delta, \phi) & := & -2\displaystyle\sum_{\substack{\mathfrak{p} \ \text{principal}}} \log N\mathfrak{p}\sum_{\substack{m \geq 1\\ (m,h_\Delta)>1}} {\frac{\widehat{\phi}\left( m \frac{\log N\mathfrak{p}}{\log \Delta}\right)}{N\mathfrak{p}^{m/2}}}.
\end{eqnarray}
Note that the proof of Lemma \ref{lem: S_1} did not actually use the non-principality of the prime ideals involved in the sum, but only the fact that the primes have principal square, as well as Lemma \ref{lem: key} and the fact that the sum began at $m = 2$. Since the principality of $\mathfrak{p}$ of course implies the principality of $\mathfrak{p}^2$, and since the condition $(m,h) > 1$ in the definition of $S_{2,2}(\Delta,\phi)$ implies that the sum again begins at least at $m = 2$, the same argument given in Lemma \ref{lem: S_1} shows that
\be\label{eq:S_22}
S_{2,2}(\Delta,\phi)\ \ll \ \displaystyle\sum_{\substack{p \\ p\,|\,\Delta}}\frac{\log p}{p}\ =\ O(\log\log\Delta).
\ee

We now analyze $S_{2,1}(\Delta,\phi)$. Note that
\begin{eqnarray}
S_{2,1}(\Delta, \phi) & \ = \ & -2\displaystyle\sum_{\substack{\mathfrak{p} \ \text{principal} \\ f_{K/\mathbb{Q}}(\mathfrak{p}) \leq 2}}\log N\mathfrak{p}\sum_{\substack{m \geq 1 \\ (m,h_{\Delta}) = 1}} {\frac{\widehat{\phi}\left( m \frac{\log N\mathfrak{p}}{\log \Delta}\right)}{N\mathfrak{p}^{m/2}}} + O(1)
\end{eqnarray}
since, as before (see Lemma \ref{lem: S_1}), the sum $\displaystyle\sum_{\substack{\mathfrak{p} \ \text{principal} \\ f_{K/\mathbb{Q}}(\mathfrak{p}) > 2}}\log N\mathfrak{p}\sum_{\substack{m \geq 1 \\ (m,h_{\Delta}) = 1}} {\frac{\widehat{\phi}\left( m \frac{\log N\mathfrak{p}}{\log \Delta}\right)}{N\mathfrak{p}^{m/2}}}$ is bounded by a convergent series. Moreover, observe that
\begin{equation}
\sum_{\substack{\mathfrak{p} \ \text{principal} \\ f_{K/\mathbb{Q}}(\mathfrak{p}) = 1}}\log N\mathfrak{p}\sum_{\substack{m \geq 1 \\ (m,h_{\Delta}) = 1}} {\frac{\widehat{\phi}\left( m \frac{\log N\mathfrak{p}}{\log \Delta}\right)}{N\mathfrak{p}^{m/2}}}\ =\ \sum_{\substack{\mathfrak{p} \ \text{principal} \\ f_{K/\mathbb{Q}}(\mathfrak{p}) = 1 \\ N\mathfrak{p} < \Delta^{\sigma}}}\log N\mathfrak{p}\sum_{\substack{m \geq 1 \\ (m,h_{\Delta}) = 1}} {\frac{\widehat{\phi}\left( m \frac{\log N\mathfrak{p}}{\log \Delta}\right)}{N\mathfrak{p}^{m/2}}},
\end{equation}
and if $\sigma < 1$ then this sum is zero for sufficiently large $\Delta$ by Lemma \ref{lem: key}.
    Thus, letting
\begin{equation}
S_{2,1}(\Delta,\phi)_2 \ = \ -2\sum_{\substack{\mathfrak{p}\ {\rm principal} \\ f_{K/\mathbb{Q}}(\mathfrak{p}) = 2}}\log N\mathfrak{p}\sum_{\substack{m \geq 1 \\ (m,h_{\Delta}) = 1}} {\frac{\widehat{\phi}\left( m\frac{\log N\mathfrak{p}}{\log\Delta}\right)}{N\mathfrak{p}^{m/2}}},
\end{equation}
we find that
\begin{equation}
S_{2,1}(\Delta,\phi)\ =\ S_{2,1}(\Delta,\phi)_{2} + O(1)
\end{equation}
and so, by \eqref{eq:diff} and \eqref{eq:S_22}, we find that
\be\label{eq:S_2}
S_{2}(\Delta,\phi)\ =\ S_{2,1}(\Delta,\phi)_{2} + O(\log\log\Delta).
\ee
\begin{prop}
We have
\begin{equation}\label{eq:expansionS21}
S_{2,1}(\Delta,\phi)_2\ =\ -2\sum_{\substack{\mathfrak{p}\ {\rm principal} \\ f_{K/\mathbb{Q}}(\mathfrak{p}) = 2}}{\frac{\log N\mathfrak{p}}{N\mathfrak{p}^{1/2}}}\widehat{\phi}\left(\frac{\log N\mathfrak{p}}{\log\Delta}\right) + O(1).
\end{equation}
\label{prop: 6.1}
\end{prop}

\begin{proof}
Let $A(\Delta,\phi)$ be the difference between $S_{2,1}(\Delta,\phi)_2$ and the main term on the right hand side of \eqref{eq:expansionS21}. Thus
\begin{equation}
A(\Delta,\phi)\ =\ -2\sum_{\substack{\mathfrak{p}\ {\rm principal} \\ f_{K/\mathbb{Q}}(\mathfrak{p}) = 2}}\log N\mathfrak{p}\sum_{\substack{m \geq 2 \\ (m,h_{\Delta}) = 1}} {\frac{\widehat{\phi}\left(m\frac{\log N\mathfrak{p}}{\log\Delta}\right)}{N\mathfrak{p}^{m/2}}}.
\end{equation}
Since $\widehat{\phi}$ is bounded and $N\mathfrak{p} \ge 2$, we have
\begin{eqnarray}
A(\Delta,\phi) &\ \ll\ &\sum_{\substack{\mathfrak{p}\ {\rm principal} \\ f_{K/\mathbb{Q}}(\mathfrak{p}) = 2}}\log N\mathfrak{p}\sum_{m=2}^{\infty} {\frac{1}{N\mathfrak{p}^{m/2}}} \nonumber\\
& \ll & \sum_{\substack{\mathfrak{p} \ {\rm principal} \\ f_{K/\mathbb{Q}}(\mathfrak{p}) = 2}}\frac{\log N\mathfrak{p}}{N\mathfrak{p}},
\end{eqnarray}
where the last statement is derived by summing the geometric series. Since each rational prime $p$ has at most $N$ prime ideals of degree 2 lying above it in $K$, we find
\begin{equation}
A(\Delta,\phi)\ \ll\ \sum_{p}\frac{N\log p}{p^2}.
\end{equation}
This sum is convergent, since it is dominated by a convergent series. Hence $A(\Delta,\phi) = O(1)$ as claimed.
\end{proof}

We now express $S_{2,1}(\Delta,\phi)_2$ in terms of primes of $K^+$.

\begin{prop} We have
\begin{equation}\label{eq:s21deltaphiqinert}
S_{2,1}(\Delta,\phi)_2 \ =\ -2\Bigg[ 2\sum_{\substack{\mathfrak{q} \subset \mathcal{O}_{K^+} \\ \mathfrak{q}\ {\rm inert\ in\ }K \\ f_{{K^+}/\mathbb{Q}}(\mathfrak{q})=1}}\frac{\log N\mathfrak{q}}{N\mathfrak{q}}\widehat{\phi}\left(2\frac{\log N\mathfrak{q}}{\log\Delta}\right)\Bigg] + O(\log\log\Delta).
\end{equation}
\label{prop: 6.2}
\end{prop}

\begin{proof}
Let $M(\Delta,\phi)$ be the main term in the expression for $S_{2,1}(\Delta,\phi)_2$ given by Proposition \ref{prop: 6.1}:
\begin{equation}
M(\Delta,\phi) \ =\ -2\sum_{\substack{\mathfrak{p} \ \text{principal} \\ f_{K/\mathbb{Q}}(\mathfrak{p}) = 2}}{\frac{\log N\mathfrak{p}}{N\mathfrak{p}^{1/2}}}\widehat{\phi}\left(\frac{\log N\mathfrak{p}}{\log\Delta}\right).
\end{equation}
Divide this sum by degree over $K^+$:
\begin{eqnarray}
M(\Delta,\phi) & \ = \ & -2\Bigg[\sum_{\substack{\mathfrak{p} \ \text{principal} \\ f_{K/K^{+}}(\mathfrak{p})=f_{K/\mathbb{Q}}(\mathfrak{p}) = 2}}{\frac{\log N\mathfrak{p}}{N\mathfrak{p}^{1/2}}}\widehat{\phi}\left( \frac{\log N\mathfrak{p}}{\log\Delta}\right) \nonumber\\ & & \ \ \ +\ \sum_{\substack{\mathfrak{p} \ \text{principal} \\ f_{K/K^{+}}(\mathfrak{p})=1, f_{K/\mathbb{Q}}(\mathfrak{p}) = 2}}{\frac{\log N\mathfrak{p}}{N\mathfrak{p}^{1/2}}}\widehat{\phi}\left( \frac{\log N\mathfrak{p}}{\log\Delta}\right)\Bigg] \nonumber\\ & \ := \ & M_1(\Delta,\phi) + M_2(\Delta,\phi). \ \ \ \ \
\end{eqnarray}

For $M_2(\Delta,\phi)$, $f_{K/K^{+}}(\mathfrak{p})=1$ implies that $\mathfrak{q} = \mathfrak{p} \cap \mathcal{O}_{K^+}$ either splits or is ramified in $K$. It follows as before from Lemma \ref{lem: key} that the contribution from split primes is zero for large enough $\Delta$ as ${\rm supp}(\hphi) \subset (-1, 1)$. The contribution from those $\mathfrak{p}$ which lie over ramified primes in $K^+$ and for which $f_{K/\mathbb{Q}}(\mathfrak{p})=2$ is bounded (up to a constant) by
\begin{equation}
\sum_{p\,|\,\Delta}\frac{\log p}{p}\ \ll\ \log\log\Delta.
\end{equation}
Therefore $M_2(\Delta,\phi) = O(\log\log\Delta)$.

Denote the main term in \eqref{eq:s21deltaphiqinert} by $M'(\Delta,\phi)$, so
\begin{equation}
M'(\Delta,\phi) \ :=\ -2\left[2 \sum_{\substack{\mathfrak{q} \subset \mathcal{O}_{K^+} \\ \mathfrak{q}\ \text{inert in K} \\ f_{{K^+}/\mathbb{Q}}(\mathfrak{q})= 1}}\frac{\log N\mathfrak{q}}{N\mathfrak{q}}\widehat{\phi}\left(2\frac{\log N\mathfrak{q}}{\log\Delta}\right)\right].
\end{equation} As $M_2(\Delta,\phi) = O(\log\log\Delta)$ it suffices to show $M'(\Delta,\phi) = M_1(\Delta,\phi)$ to complete the proof.

Let $\mathfrak{q}$ be a prime of $K^+$ of degree 1 over $\mathbb{Q}$ that is inert in $K$. Then, since $h_{K^+} = 1$, $\mathfrak{p} = \mathfrak{q}\mathcal{O}_{K}$ is principal. Moreover, $f_{K/K^+}(\mathfrak{p}) = f_{K/\mathbb{Q}}(\mathfrak{p}) = 2$ and $N\mathfrak{p} = N\mathfrak{q}^2$. Conversely, if $\mathfrak{p}$ is a prime of $K$ such that $f_{K/K^+}(\mathfrak{p}) = f_{K/\mathbb{Q}}(\mathfrak{p}) = 2$, then $\mathfrak{q} = \mathfrak{p} \cap \mathcal{O}_{K^+}$ has degree 1 over $\mathbb{Q}$ and is inert in $K$. Therefore
\begin{eqnarray}
M'(\Delta,\phi) & \ = \ & -2\Bigg[2\sum_{\substack{\mathfrak{p} \subset\mathcal{O}_{K} \ \text{principal} \\ f_{K/K^{+}}(\mathfrak{p})=f_{K/\mathbb{Q}}(\mathfrak{p}) = 2}}{\frac{\log (N\mathfrak{p}^{1/2})}{N\mathfrak{p}^{1/2}}}\widehat{\phi}\left(2 \frac{\log (N\mathfrak{p}^{1/2})}{\log\Delta}\right)\Bigg]\nonumber\\
&=& -2 \Bigg[\sum_{\substack{\mathfrak{p} \ \text{principal} \\ f_{K/K^{+}}(\mathfrak{p})=f_{K/\mathbb{Q}}(\mathfrak{p}) = 2}}{\frac{\log N\mathfrak{p}}{N\mathfrak{p}^{1/2}}}\widehat{\phi}\left(\frac{\log N\mathfrak{p}}{\log\Delta}\right)\Bigg] = M_{1}(\Delta,\phi).\ \ \ 
\end{eqnarray}
Hence, $S_{2,1}(\Delta,\phi)_2 = M(\Delta,\phi) + O(1) = M_{1}(\Delta,\phi) + M_{2}(\Delta,\phi) + O(1) = M'(\Delta,\phi) + O(\log\log\Delta)$, as claimed.
\end{proof}

\subsection{Evaluation of $S_2$}\label{sec:S_2}

We now complete the analysis of $S_2$. Let $\chi$ be the unique non-trivial character of $G := \text{Gal}(K/K^{+})$. For $\mathfrak{q}$ a prime of $K^{+}$ unramified in $K$, define $\chi(\mathfrak{q}) := \chi\left(\left(\displaystyle\frac{\mathfrak{q}}{K/K^{+}}\right)\right)$ where $\left(\displaystyle\frac{\mathfrak{q}}{K/K^{+}}\right)$ is the Artin symbol. Thus
\begin{equation}
\chi(\mathfrak{q}) = \left\{ \begin{array}{rl}
  -1 &\text{if} \ \mathfrak{q} \ \text{is inert in} \ K \\
  1 &\text{if} \ \mathfrak{q} \ \text{splits in} \ K.
       \end{array} \right.
\end{equation}
The Artin L-function associated to $\chi$ is
\begin{equation}
L(s,\chi)\ =\ \prod_{\substack{\mathfrak{q} \ \text{unramified in $K$}}}\left(1 - \frac{\chi(\mathfrak{q})}{N\mathfrak{q}^s}\right)^{-1}.
\end{equation}
Since $\chi$ is the character of a non-trivial one-dimensional representation of $G$, $L(s,\chi)$ is entire and has no zeros on the line $\Re s=1$. Define a function $U(s)$ by
\begin{equation}
U(s)\ =\ (s-1)\frac{\zeta_{K^+}(s)}{L(s,\chi)\zeta_{\text{ram}}(s)}.
\end{equation}
Here $\zeta_{\text{ram}}(s)$ is given by the partial Euler product for $\zeta_{K^+}(s)$ restricted to those primes which ramify in $K$. One has (\cite{La1}, pg. 161) that $\zeta_{K^+}(s)$ is analytic for $\Re s > 1 - 1/N$ except for a simple pole at $s=1$. Since the factor of $(s-1)$ cancels this pole, $U(s)$ is analytic for $\Re s > 1 - 1/N$. In this region, we have
\begin{equation}\label{eq:Uprod}
U(s)\ =\ (s-1)\prod_{\mathfrak{q} \ \text{inert in $K$}}\left(\frac{N\mathfrak{q}^{s}-1}{N\mathfrak{q}^{s}+1}\right)^{-1}.
\end{equation}
Therefore, for $\Re s > 1 - 1/N$ one has
\begin{equation}\label{eq:uprimeu}
\frac{U'}{U}(s)\ =\ \frac{1}{s-1} -2 \sum_{\mathfrak{q} \ \text{inert in $K$}}\sum_{m=0}^{\infty}\frac{\log N\mathfrak{q}}{(N\mathfrak{q}^s)^{2m+1}}.
\end{equation}
Consider the integral
\begin{equation}
\int_{-\infty}^{\infty}\phi(x)\frac{U'}{U}\left(1 + \frac{4\pi ix}{\log\Delta}\right)dx.
\end{equation}
We substitute the expansion from \eqref{eq:uprimeu} above. The first piece is the integral \be \int_{-\infty}^{\infty} \phi(x) \frac{\log \Delta dx}{4\pi i x}  \ = \ \frac{\log \Delta}2 \frac1{2\pi i} \int_{-\infty}^{\infty} \phi(x) \frac{dx}{x}, \ee which is just $\frac14 \phi(0)\log \Delta$ from complex analysis.\footnote{Remember that $\phi$ is an even function. The extra factor of $1/2$ is due to the pole lying on the line of integration.} The second piece becomes the integral of $\phi(x)$ against factors such as $(N\mathfrak{q})^{s(2m+1)}$ with $s = 1 + \frac{4\pi i x}{\log \Delta}$. The integration against $x$ gives the Fourier transform of $\phi$. Specifically, these terms contribute
\begin{equation}\label{eq:U'/U}
\frac{1}{4}\phi(0)\log\Delta - 2\sum_{\mathfrak{q} \ \text{inert in $K$}}\sum_{m=0}^{\infty}\frac{\log N\mathfrak{q}}{N\mathfrak{q}^{2m+1}}\widehat{\phi}\left(2(2m+1)\frac{\log N\mathfrak{q}}{\log\Delta}\right),
\end{equation}
where $\displaystyle\frac{1}{4}\phi(0)\log\Delta$ appears as half the residue of $\displaystyle\frac{1}{2}\phi(s)s^{-1}\log\Delta$ at $s=0$. Similarly to the above, one has
\begin{eqnarray}
&  & \sum_{\mathfrak{q} \ \text{inert in $K$}}\sum_{m=0}^{\infty}\frac{\log N\mathfrak{q}}{N\mathfrak{q}^{2m+1}}\widehat{\phi}\left(2(2m+1)\frac{\log N\mathfrak{q}}{\log\Delta}\right)\nonumber\\  &= \ & \sum_{\mathfrak{q} \ \text{inert in $K$}}\frac{\log N\mathfrak{q}}{N\mathfrak{q}}\widehat{\phi}\left(2\frac{\log N\mathfrak{q}}{\log\Delta}\right) + O(1) \nonumber\\
&= \ & \sum_{\substack{\mathfrak{q} \ \text{inert in $K$} \\ f_{K^{+}/\mathbb{Q}}(\mathfrak{q})=1}}\frac{\log N\mathfrak{q}}{N\mathfrak{q}}\widehat{\phi}\left(2\frac{\log N\mathfrak{q}}{\log\Delta}\right) + O(1).
\end{eqnarray}
Therefore, by Proposition \ref{prop: 6.2}, we have shown

\begin{lem}
\begin{equation}
S_{2,1}(\Delta,\phi)_2\ =\  -\frac{1}{2}\phi(0)\log\Delta + 2\int_{-\infty}^{\infty}\phi(x)\frac{U'}{U}\left(1 + \frac{4\pi ix}{\log\Delta}\right)dx + O(\log\log\Delta).
\end{equation}
\end{lem}

Write
\begin{equation}\label{eq:U}
\frac{U'}{U}(s)\ =\ \frac{1}{s-1} + \frac{\zeta_{K^+}'}{\zeta_{K^+}}(s) - \frac{L'}{L}(s,\chi) - \frac{\zeta_{\text{ram}}'}{\zeta_{\text{ram}}}(s).
\end{equation}

We have the following important fact (Theorem 5.17 of \cite{IK}).

\begin{thm}
Assume the Generalized Riemann Hypothesis. Let $L(s,\rho)$ be the Artin $L$-function associated to a (possibly trivial) one-dimensional representation $\rho$ of $G$. Let $r$ be the order of the pole of this $L$-function at $s=1$, and let $\mathfrak{q}(\chi,s)$ be the analytic conductor of the associated Hecke character. Then
\begin{equation}
-\frac{L'}{L}(1+it,\rho)\ =\ \frac{r}{s-1} + O(\log\log\mathfrak{q}(\chi,s)),
\end{equation}
the implied constant being absolute.
\label{thm: IK}
\end{thm}

In our situation, we have a factorization of the Dedekind zeta-function of $K$ just as in the case of imaginary quadratic fields:
\begin{equation}
\zeta_{K}(s)\ =\ \zeta_{K^+}(s)L(s,\chi),
\end{equation}
which may be proven by checking the local factors at each prime ideal of $K$. Thus every rational prime dividing $q(\chi)$ (the ordinary conductor) must also divide $\Delta$. But we also have $q(\chi) = |D_{K^{+}/\mathbb{Q}}|N^{K^+}_{\mathbb{Q}}\mathfrak{f}(\chi)$ for an integral ideal $\mathfrak{f}(\chi)$ of $K^+$ (\cite{IK}, pg. 142), and since each prime in the factorization of this ideal has degree at most $N$ over $\mathbb{Q}$, we find $q(\chi) \leq |D_{K^{+}/\mathbb{Q}}|\Delta^{N}$. Thus, since $|D_{K^{+}/\mathbb{Q}}|$ is independent of $\Delta$, we find  $\mathfrak{q}(\chi,s) \ll \Delta^{N}|s|^{2N}$. Since $L(s,\chi)$ is entire, we therefore obtain by Theorem \ref{thm: IK} the estimates
\begin{eqnarray}\label{eq:estimates}
-\frac{\zeta_{K^+}'}{\zeta_{K^+}}(1+it) & \ =\ & \frac{1}{s-1} + O(\log\log(\Delta|t|^{2N})) \nonumber\\
-\frac{L'}{L}(1+it,\chi) &\ \ll \ & \log\log(\Delta^{N}|t|^{2N}).
\end{eqnarray}
Combining these estimates with the fact that
\begin{equation}
\frac{\zeta_{\text{ram}}'}{\zeta_{\text{ram}}}(1+it)\ \ll\ \log\log\Delta
\end{equation}
(use $\displaystyle\sum_{p\,|\,\Delta}\frac{\log p}{p}\ \ll\ \log\log\Delta$), one finds since $\phi$ is Schwartz that
\begin{equation}
\int_{-\infty}^{\infty}\phi(x)\frac{U'}{U}\left(1 + \frac{4\pi ix}{\log\Delta}\right)dx\ \ll\ \log\log\Delta,
\end{equation}
where the implied constant depends only on $\phi$ and $N$. Combined with the previous lemma, this proves

\begin{lem}\label{lem: S_2}
We have
\be
S_{2,1}(\Delta,\phi)_2\ =\  -\frac{1}{2}\phi(0)\log\Delta + O(\log\log\Delta).
\ee
Thus, by \eqref{eq:S_2}, we have
\be
S_{2}(\Delta,\phi) = -\frac{1}{2}\phi(0)\log\Delta + O(\log\log\Delta)
\ee
as well.
\end{lem}

We are now ready to prove the main theorem.

\subsection{Proof of Theorem \ref{thm:main}}

Our main result trivially follows from our analysis of $S_1$ and $S_2$.

\begin{proof}[Proof of Theorem \ref{thm:main}] By \eqref{eq:density}, we have
\begin{eqnarray}
 D_{\widehat{\mathcal{CL}(K)}}(\phi)& \ = \ & \frac{1}{\log \Delta}
\Bigg[ \frac{4}{h_{\Delta}}\int_{0}^{\infty}{\widehat{\phi}\left(\frac{x}{\log \Delta}\right)\cosh\left(\frac{x}{2}\right)dx}\nonumber\\
& & \ +\  \widehat{\phi}(0)\cdot\Big(\log \Delta - 2N\gamma_{{\rm EM}} -  2N\log 8\pi\Big)
+ S_{1}(\Delta, \phi) + S_{2}(\Delta, \phi)\nonumber\\ & & \  + 2N\int_{0}^{\infty}{\frac{\widehat{\phi}(0) - \widehat{\phi}(x)}{2\sinh(x/2)}dx} \Bigg].
\end{eqnarray}
By Lemmas \ref{lem:htrigfns}, \ref{lem: S_1} and \ref{lem: S_2}, and since $N$ is fixed and $r_{1} \leq N$, this entire expression equals
\begin{equation}
\frac{1}{\log\Delta}\left[\widehat{\phi}(0)\log\Delta - \frac{1}{2}\phi(0)\log\Delta + O(\log\log\Delta)\right],
\end{equation}
which completes the proof.
\end{proof}

\section{Lower Order Terms}\label{sec:3}

In this section, we prove Theorem \ref{thm:lowterm}, which gives the lower order terms for a sub-family of our original family. Similar to investigations of the 1-level density in other families (such as \cite{Gao,Mil4}), we are able to isolate lower order terms if we restrict to a sub-family which simplifies some of the terms. To derive the lower order terms of the 1-level density, we make the additional assumption that the class number of $K_0$ in the narrow sense is 1 (recall that the {\em narrow class group} of $K_0$ is defined similarly to the ordinary ideal class group, except that ideals are considered equivalent if and only if they differ by a totally positive element of $K_0$ rather than an arbitrary one). We will make use of the following facts, which rephrase Theorems 1 and 2 of \cite{H}.

\begin{prop}\label{prop:infinitelymany}
The family $\{K_{\Delta}\}$ of CM-fields for which $K^{+} = K_{0}$ contains infinitely many fields of odd class number (in the usual sense).
\end{prop}

Thus we may consider $\{K_{\Delta} : 2\nmid h_{\Delta}\}$ as a sub-family of $\{K_{\Delta}\}$.\\ \ \\

\texttt{\emph{Unless otherwise stated, $K = K_{\Delta}$ denotes a CM-field of odd class number such that $K^{+} = K_{0}$.}}\\ \ \\

\begin{prop}\label{prop:oddclassnumber}
Let $K$ be a CM-field such that $K^{+}$ has class number 1, and suppose that the class number of $K$ is odd. Then at most one finite prime of $K^{+}$ ramifies in $K$.
\end{prop}

Writing $K = K^{+}(\sqrt{\beta})$, this implies that the relative discriminant $D(K/K^{+})$ is divisible by at most one prime of $\mathcal{O}_{K^{+}}$, which we denote $\mathfrak{q}_{K/K^{+}} = \mathfrak{q}$. Since the CM-fields $K$ for which $\mathcal{O}_{K} = \mathcal{O}_{K^+}[\sqrt{\beta}]$ have discriminant $(4\beta)$, which is divisible by more than one prime, any $K$ as in the proposition must have ring of integers $\mathcal{O}_{K} = \mathcal{O}_{K^+}\left[\displaystyle\frac{1 + \sqrt{\beta}}{2}\right]$ and relative discriminant $D_{K/K^{+}} = (\beta)$. Since $\beta$ is square-free and $h_{K^+}=1$, the proposition then implies that $D_{K/K^{+}}$ is prime. Arguing as in the end of the proof of Lemma \ref{lem: key}, we moreover have
\begin{equation}
N^{K^{+}}_{\mathbb{Q}}(D_{K/K^{+}})\ =\ |N^{K^{+}}_{\mathbb{Q}}(\beta)|\ =\ \frac{\Delta}{D^{2}_{K^{+}/\mathbb{Q}}}.
\end{equation}
Thus the contribution from the ramified prime of $\mathcal{O}_{K^+}$ to terms like $\displaystyle\frac{\log N\mathfrak{q}}{N\mathfrak{q}^{1/2}}$ is $\displaystyle O\left(\frac{\log\Delta}{\Delta^{1/2}}\right)$, where the implied constant depends only on $K^{+} = K_{0}$. Since we're only interested in terms of size $\frac{1}{\log\Delta}$, we may therefore ignore the ramified prime in what follows. \\

\subsection{Evaluation of $S_{1}$ (Redux)}

With all notation as before, we again consider $S_{1}(\Delta, \phi)$. Our goal is to improve the calculation to terms of size $1/\log\Delta$. Recall (cf. \eqref{eq:s1deltaphiprincnonprinc}) that
\begin{equation}
S_{1}(\Delta, \phi)\ =\ -2\sum_{\substack{\mathfrak{p}\ {\rm non-principal} \\ \mathfrak{p}^2\ {\rm principal}}}\frac{\log N\mathfrak{p}}{N\mathfrak{p}}\widehat{\phi}\left(2\frac{\log N\mathfrak{p}}{\log \Delta}\right) + O(1).
\ee
\emph{Since now the class number of $K$ is odd}, no non-principal prime has principal square, so in fact
\be
S_{1}(\Delta, \phi)\ =\ -2\sum_{\mathfrak{p}\ {\rm non-principal}}\log N\mathfrak{p}\sum_{\substack{m \geq 3 \\ \mathfrak{p}^{m} \ {\rm principal}}}\frac{\widehat{\phi}\left( m \frac{\log N\mathfrak{p}}{\log \Delta}\right)}{N\mathfrak{p}^{m/2}}.
\ee

Observe that if $\mathfrak{p}$ is non-principal, then $f_{K/K^{+}}(\mathfrak{p}) = 1$, since otherwise $\mathfrak{p}$ lies over an inert prime of $K^{+}$ and so must be principal since $h_{K^+}=1$. Let $m > 1$ be an integer such that $\mathfrak{p}^m$ is principal. Let $\mathfrak{p}^{m} = \alpha\mathcal{O}_{K}$, and suppose $\alpha \in \mathcal{O}_{K^{+}}$. Then $N^{K}_{K^+}(\mathfrak{p}^{m}) = (\alpha^{2})$. Since $f_{K/K^{+}}(\mathfrak{p}) = 1$, the ideal $\mathfrak{q} = N^{K}_{K^+}(\mathfrak{p})$ of $\mathcal{O}_{K^+}$ is prime, so unique factorization into primes implies that $m$ must be even. Consequently, since the fact that $h_K$ is odd implies that the order $d$ of $\mathfrak{p}$ in $\mathcal{CL}(K)$ must be odd as well, we must have $\alpha \in \mathcal{O}_{K}\backslash\mathcal{O}_{K^+}$ if $\mathfrak{p}^{d} = (\alpha)$. Hence, we may write $\alpha = x + y\sqrt{\beta}$, where $x,y \in \mathcal{O}_{K^+}$ and $y \neq 0$. Thus
\be
N^{K}_{\mathbb{Q}}(\mathfrak{p}^{d})\ =\ |N^{K}_{\mathbb{Q}}(\alpha)|,
\ee
so the proof of Lemma \ref{lem: key} implies that
\begin{equation}\label{eq:lowerbound}
N^{K}_{\mathbb{Q}}(\mathfrak{p})\ \geq\ (C\Delta)^{1/d},
\end{equation}
where $C$ depends only on $K^+ = K_0$.\\
\indent Since $\mathfrak{p}^m$ is principal if and only if $d|m$, we have (writing $d = d_{\mathfrak{p}}$ to specify the prime),
\begin{eqnarray}
S_{1}(\Delta, \phi) &\ =\ -2\displaystyle\sum_{\mathfrak{p}\ {\rm non-principal}}\log N\mathfrak{p}\sum_{k=1}^{\infty}\frac{\widehat{\phi}\left( d_{\mathfrak{p}}k \frac{\log N\mathfrak{p}}{\log \Delta}\right)}{N\mathfrak{p}^{d_{\mathfrak{p}}k/2}} \nonumber\\
&\ =\ -2\displaystyle\sum_{\substack{\mathfrak{p}\ {\rm non-principal} \\ \log N\mathfrak{p} < \frac{\sigma\log\Delta}{d_{\mathfrak{p}}}}}\log N\mathfrak{p}\sum_{k=1}^{\infty}\frac{\widehat{\phi}\left( d_{\mathfrak{p}}k \frac{\log N\mathfrak{p}}{\log \Delta}\right)}{N\mathfrak{p}^{d_{\mathfrak{p}}k/2}} \nonumber\\
\end{eqnarray}
so \eqref{eq:lowerbound} and the fact that $\sigma < 1$ imply that $S_{1}(\Delta, \phi) = 0$ for sufficiently large $\Delta$ because the sum is vacuous.

\subsection{Evaluation of $S_{2}$ (Redux)}

We have \\
\be
S_{2}(\Delta,\phi)\  =\  -2\displaystyle\sum_{\mathfrak{p} \ {\rm principal}}\log N\mathfrak{p}\sum_{m = 1}^{\infty} {\frac{\widehat{\phi}\left(m \frac{\log N\mathfrak{p}}{\log \Delta}\right)}{N\mathfrak{p}^{m/2}}}.
\ee
As argued above, the contribution from the ramified prime is negligible, while the contribution from the primes of degree 1 over $K_0$ is ultimately zero. Consequently, for $\Delta$ large enough, we have (up to the $O\left(\frac{\log\Delta}{\Delta^{1/2}}\right)$ error from the ramified prime)
\begin{eqnarray}\label{eq:s2deltaphiexpansion}
S_{2}(\Delta,\phi) &\ =\ &  -2\displaystyle\sum_{\substack{\mathfrak{p} \ {\rm principal} \\ f_{K/K_{0}}(\mathfrak{p})=2}}\log N\mathfrak{p}\sum_{m = 1}^{\infty} {\frac{\widehat{\phi}\left(m \frac{\log N\mathfrak{p}}{\log \Delta}\right)}{N\mathfrak{p}^{m/2}}} \nonumber\\
& = & -4\displaystyle\sum_{\substack{\mathfrak{q} \subset \mathcal{O}_{K_0} \\ {\rm inert \ in} \ K}}\log N\mathfrak{q}\sum_{m = 1}^{\infty} {\frac{\widehat{\phi}\left(2m \frac{\log N\mathfrak{q}}{\log \Delta}\right)}{N\mathfrak{q}^m}}.
\end{eqnarray}
Recall from Section \ref{sec:S_2} that
\bea & &
\displaystyle\int_{-\infty}^{\infty}\phi(x)\frac{U'}{U}\left(1 + \frac{4\pi ix}{\log\Delta}\right)dx\nonumber\\ & = \ & \frac{1}{4}\phi(0)\log\Delta - 2\sum_{\substack{\mathfrak{q}\subset\mathcal{O}_{K_0} \\ {\rm inert \ in} \ K}}\sum_{\substack{m\geq 1 \\ {\rm odd}}}\frac{\log N\mathfrak{q}}{N\mathfrak{q}^{m}}\widehat{\phi}\left(2m\frac{\log N\mathfrak{q}}{\log\Delta}\right).
\eea
Thus, using Lemma \ref{lem: key} and the fact that contribution from the ramified prime is negligible, we have by the compact support of $\widehat{\phi}$
\begin{align}
S_{2}(\Delta,\phi) &\ =\ -4\displaystyle\sum_{\substack{\mathfrak{q} \subset \mathcal{O}_{K_0} \\ {\rm inert \ in} \ K}}\log N\mathfrak{q}\sum_{m = 1}^{\infty} {\frac{\widehat{\phi}\left(2m \frac{\log N\mathfrak{q}}{\log \Delta}\right)}{N\mathfrak{q}^m}} \nonumber\\
&\ =\ -\frac{1}{2}\phi(0)\log\Delta + 2\displaystyle\int_{-\infty}^{\infty}\phi(x)\frac{U'}{U}\left(1 + \frac{4\pi ix}{\log\Delta}\right)dx \nonumber\\
& \hspace{0.5in} -\ 4\displaystyle\sum_{\substack{\mathfrak{q} \subset \mathcal{O}_{K_0} \\ {\rm inert \ in} \ K}}\log N\mathfrak{q}\sum_{\substack{m \geq 2 \\ {\rm even}}}{\frac{\widehat{\phi}\left(2m \frac{\log N\mathfrak{q}}{\log \Delta}\right)}{N\mathfrak{q}^m}}.
\end{align}
Therefore, to complete the analysis of the lower-order terms, we must show that
\be
2\displaystyle\int_{-\infty}^{\infty}\phi(x)\frac{U'}{U}\left(1 + \frac{4\pi ix}{\log\Delta}\right)dx - 4\displaystyle\sum_{\substack{\mathfrak{q} \subset \mathcal{O}_{K_0} \\ {\rm inert \ in} \ K}}\log N\mathfrak{q}\sum_{\substack{m \geq 2 \\ {\rm even}}}{\frac{\widehat{\phi}\left(2m \frac{\log N\mathfrak{q}}{\log \Delta}\right)}{N\mathfrak{q}^m}}
\ee equals $c_K+o(1)$, with $c_K$ bounded independently of $K$. Note that in the explicit formula the terms $S_1(\Delta,\phi)$ and $S_2(\Delta,\phi)$ are multiplied by $1/\log \Delta$; thus if we show the term above is $c_K+o(1)$, we will have isolated its contribution to the first lower order term.

First, note that since the compact support of $\widehat{\phi}$ restricts the sums to be finite, we have using Taylor series
\begin{align}
\displaystyle\sum_{\substack{\mathfrak{q} \subset \mathcal{O}_{K_0} \\ {\rm inert \ in} \ K}}\log N\mathfrak{q}\sum_{\substack{m \geq 2 \\ {\rm even}}}{\frac{\widehat{\phi}\left(2m \frac{\log N\mathfrak{q}}{\log \Delta}\right)}{N\mathfrak{q}^m}}
&\ =\ \widehat{\phi}(0)\displaystyle\sum_{\substack{\mathfrak{q} \subset \mathcal{O}_{K_0} \\ {\rm inert \ in} \ K}}\log N\mathfrak{q}\sum_{\substack{m \geq 2 \\ {\rm even}}}\frac{1}{N\mathfrak{q}^m} + O\left(\frac{1}{\log\Delta}\right) \nonumber\\
&\ =\ \widehat{\phi}(0)\displaystyle\sum_{\substack{\mathfrak{q} \subset \mathcal{O}_{K_0} \\ {\rm inert \ in} \ K}}\frac{\log N\mathfrak{q}}{N\mathfrak{q}^{2}-1} + O\left(\frac{1}{\log\Delta}\right) \nonumber\\
\end{align}
and since each prime of $\mathcal{O}_{K_0}$ lies over at most $N$ rational primes, this is dominated by a convergent $p$-series independent of $K$, and thus is $O(1)$. \\
\indent To analyze the integral of $\phi$ against the logarithmic derivative of $U(s)$, let $\beta_{k}(\Delta)$ denote the $k$-th coefficient in the power series expansion of the logarithmic derivative of $U(s)$ about $s=1$; thus
\be
\frac{U'}{U}\left(1 + \frac{4\pi ix}{\log\Delta}\right)\ =\ \frac{\log\Delta}{4\pi ix} + \sum_{k=0}^{\infty}\beta_{k}(\Delta)\left(\frac{4\pi ix}{\log\Delta}\right)^{k}.
\ee
To get rid of the term $\log\Delta/4\pi ix$, observe that $\Im \frac{U'}{U}\left(1 + 4\pi ix/\log\Delta\right)$ is an odd function of $x$, so that
\be
\int_{-\infty}^{\infty}\phi(x)\frac{U'}{U}\left(1 + \frac{4\pi ix}{\log\Delta}\right)dx\ =\ \int_{-\infty}^{\infty}\phi(x)\Re\frac{U'}{U}\left(1 + \frac{4\pi ix}{\log\Delta}\right)dx
\ee
and
\be
\Re\frac{U'}{U}\left(1 + \frac{4\pi ix}{\log\Delta}\right)\ =\ \sum_{k=0}^{\infty}\beta_{2k}(\Delta)\left(\frac{4\pi ix}{\log\Delta}\right)^{2k}.
\ee
Recall from \ref{sec:S_2} that
\be
U(s)\ =\ (s-1)\frac{\zeta_{K_0}(s)}{L(s,\chi)\zeta_{\text{ram}}(s)}
\ee
and that $U(s)$ is analytic and non-zero at $s=1$. A straightforward computation, using the fact that $L(s,\chi) = \zeta_{K}(s)/\zeta_{K_0}(s)$, then yields
\begin{equation}\label{eq:beta_0}
\beta_{0}(\Delta) = \frac{U'}{U}(1)\ =\ 2\frac{\gamma_{K_0}}{\rho_{K_0}} - \frac{\gamma_K}{\rho_K} + O\left(\frac{1}{\log\Delta}\right)
\end{equation}
where for a number field $E/\mathbb{Q}$, $\rho_E$ is the residue of its Dedekind zeta function at the simple pole $s=1$
\be
\rho_{E}\ =\ {\rm res}_{s=1}\zeta_{E}(s)\ =\ \frac{2^{r_1}(2\pi)^{r_2}h_{E}R_{E}}{w_{E}\sqrt{|D_{E/\mathbb{Q}}|}}
\ee
and $\gamma_E$ denotes its Euler constant
\be
\gamma_E\ =\ \frac{d}{ds}[(s-1)\zeta_{E}(s)]_{s=1}\ =\ \lim_{s\to1}\left(\zeta_{E}(s) - \frac{\rho_E}{s-1}\right).
\ee
The $O(1/\log\Delta)$ term in \eqref{eq:beta_0} comes from $\zeta_{{\rm ram}}(s)$. We claim that $\beta_{0}(\Delta) = O(1)$ as $\Delta\to\infty$, with the implied constant depending only on $K_0$.\\
\indent We use the following bound for the number field Euler constant, which is Theorem 7 of \cite{MO}. Let $E$ be a number field of degree $n$ over $\mathbb{Q}$, with $r_1$ real and $2r_2$ complex embeddings. Denote the embeddings $K \hookrightarrow K^{(i)}$, and arrange them in such a way that $K \hookrightarrow K^{(i)}$ is real for $1 \leq i \leq r_1$, imaginary for $r_{1} + 1 \leq i \leq r_{1} + r_{2}$, and $\overline{K^{(i + r_{2})}} = K^{(i)}$. Let $\epsilon_{1},...,\epsilon_{r}$ be an independent set of generators for the unit group of $\mathcal{O}_{E}$ modulo roots of unity, where $r = r_{1} + r_{2} - 1$. Let $M$ be the largest of the values $|\log|\epsilon^{(i)}_{j}||$ for $1 \leq i,j \leq r$. Also, choose an integral basis $\beta_{1},...,\beta_{n}$ for $\mathcal{O}_{E}$ over $\mathbb{Q}$, and let $(\gamma_{ij})$ be the inverse of the non-singular matrix $(\beta_{j}^{(i)})$. Finally, set $\gamma = \max_{i,j}|\gamma_{ij}|$. Then we have
\begin{prop}\label{prop:eulerbound}
\be
|\gamma_E|\ \leq\ \rho_{E}(1 + n2^{n}\max(1,\Phi_{0}^{n}))
\ee
where $\Phi_{0} = 2^{n-1}n^{2n}\gamma^{n-1}e^{rM(n-1)}$.
\end{prop}
In our setting (e.g. CM-fields of odd class number over a fixed totally real field of strict class number 1), the values $\gamma$ and $M$, which a priori depend on $K = K_{\Delta}$, can in fact be made independent of $\Delta$ (see Appendix \ref{sec:appendix} for justification). Combining this fact with the above proposition and \eqref{eq:beta_0}, as well as the fact that $n = [K:\mathbb{Q}] = 2N$ is fixed, we find
\begin{eqnarray}
\beta_{0}(\Delta) & = & 2\frac{\gamma_{K_0}}{\rho_{K_0}} - \frac{\gamma_K}{\rho_K} + O\left(\frac{1}{\log\Delta}\right) \nonumber\\
& \ll & 2\frac{\gamma_{K_0}}{\rho_{K_0}} + \frac{\rho_{K}(1 + 2N2^{2N}\max(1,\Phi_{0}^{2N}))}{\rho_K} + O\left(\frac{1}{\log\Delta}\right) \nonumber\\
& = &  2\frac{\gamma_{K_0}}{\rho_{K_0}} + 1 + 2N2^{2N}\max(1,\Phi_{0}^{2N}) + O\left(\frac{1}{\log\Delta}\right) \nonumber\\
& = & O(1)
\end{eqnarray}
with the implied constant depending only on $K_0$. \\
\indent Now,
\begin{eqnarray}
\int_{-\infty}^{\infty}\phi(x)\frac{U'}{U}\left(1 + \frac{4\pi ix}{\log\Delta}\right)dx & = & \int_{-\infty}^{\infty}\phi(x)\Re\frac{U'}{U}\left(1 + \frac{4\pi ix}{\log\Delta}\right)dx \nonumber\\
& = & \int_{-\infty}^{\infty}\phi(x)\sum_{k=0}^{\infty}\beta_{2k}(\Delta)\left(\frac{4\pi ix}{\log\Delta}\right)^{2k}dx \nonumber\\
& = & \widehat{\phi}(0)\beta_{0}(\Delta) + \int_{-\infty}^{\infty}\phi(x)\sum_{k=1}^{\infty}\beta_{2k}(\Delta)\left(\frac{4\pi ix}{\log\Delta}\right)^{2k}dx. \nonumber\\
\label{eq:int}
\end{eqnarray}
To estimate the integral, observe that
\be
\beta_{k}(\Delta)\ =\ \gamma_{k} - \gamma_{k}(\Delta) + O\left(\frac{1}{\log\Delta}\right)
\ee
where $\gamma_{k}$ and $\gamma_{k}(\Delta)$ are the coefficients in the power series expansion of the logarithmic derivative of $\zeta_{K_0}(s)$ and $L(s,\chi)$, respectively, about $s=1$. The Riemann hypothesis for $L(s,\chi)$ implies
\be
\gamma_{k}(\Delta)\ \ll\ (\log\log\Delta)^{k+1}
\ee
and therefore
\be
\beta_{k}(\Delta)\ \ll\ (\log\log\Delta)^{k+1}
\ee
with the implied constant depending on $k$ and $K_0$. Hence, from \eqref{eq:int}, we obtain
\begin{eqnarray}
\int_{-\infty}^{\infty}\phi(x)\frac{U'}{U}\left(1 + \frac{4\pi ix}{\log\Delta}\right)dx & = &
\widehat{\phi}(0)\beta_{0}(\Delta) + \int_{-\infty}^{\infty}\phi(x)\sum_{k=1}^{\infty}\beta_{2k}(\Delta)\left(\frac{4\pi ix}{\log\Delta}\right)^{2k}dx
\nonumber\\
& = & \widehat{\phi}(0)\beta_{0}(\Delta) + \sum_{k=1}^{\infty}\widehat{\phi}^{(2k)}(0)\beta_{2k}(\Delta)\left(\frac{4\pi i}{\log\Delta}\right)^{2k} \nonumber\\
& = & \widehat{\phi}(0)\beta_{0}(\Delta) + O\left(\frac{(\log\log\Delta)^3}{(\log\Delta)^{2}}\right)
\end{eqnarray}
with the implied constant depending on $\phi$ and $K_0$. Finally, combining this with the expression for the 1-level density given in \ref{thm:oneleveldensity}, we obtain the full first lower-order term, completing the proof of Theorem \ref{thm:lowterm}.

\appendix
\section{$\Delta$-independence in Proposition \ref{prop:eulerbound}}\label{sec:appendix}
The purpose of this appendix is to prove the $\Delta$-independence alluded to after Proposition \ref{prop:eulerbound}. Namely, we have
\begin{prop}
Let $K$ be a CM-field of odd class number such that $K^+$ has strict class number 1, and let the values $\gamma = \gamma(K)$ and $M = M(K)$ associated to $K$ be defined as in Proposition \ref{prop:eulerbound} (note that $\gamma(K)$ is distinct from the number field Euler constant $\gamma_K$). Then we may bound $\gamma$ and $M$ by constants depending only on $K^+$.
\end{prop}
Thus, if we begin with a totally real field $K_0$ of strict class number 1 and consider the family $\{K_\Delta\}$ of all CM-fields of odd class number for which $K^{+} = K_0$, then
\be
\gamma(K_\Delta), M(K_\Delta)\ =\ O(1) \ {\rm as} \ \Delta\to\infty
\ee
with the implied constants depending on $K_0$. Actually, this is true even when $K$ has even class number, but that doesn't matter for us since there may be too many ramified primes.
\begin{proof}
Lemma 15 of \cite{Ok} implies that if $K_0$ is a totally real field of strict class number 1, then for any CM-field $K$ with $K^{+}=K_0$, the Hasse unit index $Q_K$ satisfies
\be
Q_{K}\ =\ [\mathcal{O}_{K}^{*} : W_{K}\mathcal{O}_{K_0}^{*}]\ =\ 1,
\ee
where $W_K$ is the group of roots of unity contained in $K$. Consequently, any independent set $\epsilon_{1},...,\epsilon_{r}$ of generators for $\mathcal{O}_{K_0}^{*}$ modulo $\{\pm1\}$ also serves as independent set of generators for $\mathcal{O}_{K}^{*}$ modulo $W_K$. This, together with the exact sequence
\be
1 \to {\rm Gal}(K/K_0) \to {\rm Gal}(K/\mathbb{Q}) \to {\rm Gal}(K_{0}/\mathbb{Q}) \to 1
\ee
implies that
\[
M(K)\ =\ \max_{\substack{1 \leq j \leq r \\ \sigma\in {\rm Gal}\;(K/\mathbb{Q})}}|\log|\sigma(\epsilon_{j})||
\]
depends only on $K_0$, as desired. \\
\indent To bound $\gamma(K)$, recall that $\mathcal{O}_{K} = \mathcal{O}_{K_0}[\alpha]$, where $\alpha = (1 + \sqrt{\beta})/2$ for $\beta\in\mathcal{O}_{K_0}$ a totally negative element. Thus, if $x_{1},...,x_{N}$ is an integral basis for $\mathcal{O}_{K_0}$ over $\mathbb{Q}$, then
\be
\twocase{\beta_j \ = \ }{x_j}{if $1 \le j \le N$}{\alpha x_{j-N}}{if $N+1 \le j \le 2N$}
\ee
is an integral basis for $\mathcal{O}_{K}$ over $\mathbb{Q}$. Consequently, the matrix $(\beta_{j}^{(i)})$ takes the block form
\be
(\beta_{j}^{(i)})\ =\ \left(\begin{array}{cc} X & AX \\ X & \overline{A}X\end{array}\right)
\ee
where $X = (x_{j}^{(i)})_{1 \leq i,j \leq N}$, $A$ is the diagonal matrix
\be
A\ =\ \left(\begin{matrix}\alpha^{(1)} & \\ & \alpha^{(2)} \\ & & \ddots \\ & & & \alpha^{(N)}\end{matrix}\right),
\ee
and we've used the fact that $x_{j}^{(i)} = x_{j}^{(i+N)}$ and $\alpha^{(i)} = \overline{\alpha^{(i+N)}}$ for $1\leq i \leq N$ since $K^{(i+N)} = \overline{K^{(i)}}$ and $K_0$ is totally real. It is then straightforward to check that the inverse of $(\beta_{j}^{(i)})$ is given in block form by
\be\label{eq:gamma}
(\gamma_{ij})\ =\ \left(\begin{matrix} X^{-1}\overline{A}(\overline{A}-A)^{-1} & -X^{-1}A(\overline{A} - A)^{-1} \\ -X^{-1}(\overline{A}-A)^{-1} & X^{-1}(\overline{A}-A)^{-1} \end{matrix}\right).
\ee
Note that the invertibility of $\overline{A} - A$ follows from the fact that $\alpha^{(i)} \neq \overline{\alpha^{(i)}}$ for any $i$; indeed, $\alpha^{(i)} = (1 + \sqrt{\beta}^{(i)})/2$, and $\sqrt{\beta}^{(i)}$ is purely imaginary since $\beta$ is totally negative. Also, $X$ is invertible since the integral basis $x_{1},...,x_{N}$ is linearly independent over $\mathbb{Q}$. Consequently, to bound $\gamma = \max_{1\leq i,j\leq 2N}|\gamma_{ij}|$ solely in terms of $K_0$, it suffices to so bound the entries of each of the matrices $(\overline{A} - A)^{-1}, A(\overline{A} - A)^{-1}$, and $\overline{A}(\overline{A} - A)^{-1}$. \\
\indent Recall from the beginning of Section \ref{sec:3} that
\be
|N^{K_0}_{\mathbb{Q}}(\beta)|\ =\ \frac{\Delta}{D^{2}_{K_{0}/\mathbb{Q}}}.
\ee
Moreover, $|N^{K}_{\mathbb{Q}}(\sqrt{\beta})|\ =\ |N^{K_0}_{\mathbb{Q}}(N^{K}_{K_0}(\sqrt{\beta}))| = |N^{K_0}_{\mathbb{Q}}(\beta)|$. But by definition
\be
N^{K}_{\mathbb{Q}}(\sqrt{\beta})\ =\ \prod_{\substack{K \hookrightarrow K^{(i)} \\ 1\leq i \leq 2N}}\sqrt{\beta}^{(i)}
\ee
and since $K$ is CM, we have $\left|\sqrt{\beta}^{(i)}\right| = |\sqrt{\beta}^{(j)}|$ for all $i,j$ (cf. \cite{Wa}, pg. 38). Therefore, since $|N^{K}_{\mathbb{Q}}(\sqrt{\beta})| = \Delta/D^{2}_{K_{0}/\mathbb{Q}}$, we find that
\be
\left|\sqrt{\beta}^{(i)}\right|\ =\ \left(\frac{\Delta}{D^{2}_{K_{0}/\mathbb{Q}}}\right)^{1/2N}
\ee
for any $i$. This in fact implies the desired bound on the entries of the matrices in question: we have
\be
(\overline{A} - A)^{-1}\ =\ \left(\begin{matrix}(\overline{\alpha^{(1)}} - \alpha^{(1)})^{-1} & \\ & (\overline{\alpha^{(2)}} - \alpha^{(2)})^{-1} \\ & & \ddots \\ & & & (\overline{\alpha^{(N)}} - \alpha^{(N)})^{-1}\end{matrix}\right)
\ee
and for any $i$, we have (since $\alpha^{(i)} = (1 + \sqrt{\beta}^{(i)})/2$)
\begin{eqnarray}
|(\overline{\alpha^{(i)}} - \alpha^{(i)})^{-1}| &\ =\ & \left|\sqrt{\beta}^{(i)}\right|^{-1} \nonumber\\
& = & \left|\frac{D^{2}_{K_{0}/\mathbb{Q}}}{\Delta}\right|^{1/2N} \nonumber\\
& \leq  & |D_{K_{0}/\mathbb{Q}}|^{1/N}.
\end{eqnarray}
For the matrices $A(\overline{A} - A)^{-1}$ and $\overline{A}(\overline{A} - A)^{-1}$, we have for any $i$
\begin{eqnarray}
|\alpha^{(i)}(\overline{\alpha^{(i)}} - \alpha^{(i)})^{-1}| &\ \leq\ & \frac{1 + \left|\sqrt{\beta}^{(i)}\right|}{2\left|\sqrt{\beta}^{(i)}\right|} \nonumber\\
& \leq & \frac{1}{2} + |D_{K_{0}/\mathbb{Q}}|^{1/N}
\end{eqnarray}
and we get the same bound for the entries of $\overline{A}(\overline{A} - A)^{-1}$ since $|\alpha^{(i)}| = |\overline{\alpha^{(i)}}|$.
\end{proof}

\ \\

\end{document}